\documentclass[twoside]{article}
\usepackage[english]{babel}
\usepackage{amssymb,amsmath,amsthm}
\usepackage{color}
\newcommand{\ds}{\displaystyle}

\newtheorem{de}{Definition}[section]
\newtheorem{theor}{Theorem}[section]
\newtheorem{pr}{Proposition}[section]

\begin{document}
\begin{center}
{\Large \textbf{On low-dimensional filiform Leibniz
algebras\\ and their invariants}}\\

{\sc $^1$I. S. Rakhimov and $^2$Munther A. Hassan}

$^{1,2}$Institute for Mathematical Research (INSPEM)

$^1$Department of Mathematics, FS,\\UPM, 43400,
Serdang, Selangor Darul Ehsan, Malaysia,\\[0pt]

$^{1}$risamiddin@gmail.com \ \
$^{2}$munther\_abd@yahoo.com
\end{center}

\author{ I.S. Rakhimov$^{1}$,Munther A. Hassan $^{2},$ $}
\title{\mathbf{On low-dimensional filiform Leibniz
algebras \\ and their invariants}}

\begin{abstract} The paper deals with the complete classification of a subclass of complex filiform
Leibniz algebras in dimensions 5 and 6. This subclass arises from
the naturally graded filiform Lie algebras. We give a complete
list of algebras. In parametric families cases, the corresponding
orbit functions (invariants) are given. In discrete orbits case,
we show a representative of the orbits.
\end{abstract}



\thispagestyle{empty}

\section{Introduction}

\qquad  Leibniz algebras were introduced by J. -L. Loday \cite{L}.
A skew-symmetric Leibniz algebra is a Lie algebra. The main
motivation of Loday to introduce this class of algebras was the
search of an ``obstruction'' to the periodicity in algebraic
$K-$theory. Besides this purely algebraic motivation, some
relationships with classical geometry, non-commutative geometry,
and physics have been recently discovered. The present paper deals
with the low-dimensional case of a subclass of filiform Leibniz
algebras.

The outline of the paper is as follows. Section 1 is an
introduction to the subclass of Leibniz algebras that we are going
to investigate. The main results of the paper consisting of 
complete classification of a subclass of low dimensional filiform
Leibniz algebras are in Section 2. Here, for 5- and 6-dimensional
cases, we give lists of isomorphism classes. For parametric family
cases, the corresponding invariant functions are presented. Since
the proofs in 6-dimensional case are similar to those in
5-dimensional case, the detailed proofs are given for dimension 5
only.

\begin{de} An algebra L over a field K is called a Leibniz algebra, if its
bilinear operation $[\cdot,\cdot]$ satisfies the following Leibniz
identity:
 $ \ds [x,[y,z]]=[[x,y],z]-[[x,z],y],$\ \ \ \ \ \ \ $ x,y,z\in L. $
\end{de}

Onward, all algebras are assumed to be over the fields of complex
numbers $\mathbb{C}.$

Let $L$ be a Leibniz algebra. We put:
\begin{equation*}
L^{1}=L,\ \ L^{k+1}=[L^{k},L],\ k\geq 1.
\end{equation*}
\begin{de}
A Leibniz algebra L is said to be nilpotent, if there exists $s\in
\mathbb{N},$ such that
\begin{equation*}L^{1}\supset  L^{2}\supset ...\supset L^{s}=\{0\}.\end{equation*}
\end{de}
\begin{de}
A Leibniz algebra $L$ is said to be filiform, if $ \ds dimL^{i}=n-i,$ where $%
n=dimL$ and $2\leq i\leq n.$
\end{de}

The set of $n-$dimensional filiform
Leibniz algebras we denote by $Leib_n.$

  The following theorem from \cite{GO} splits $Leib_{n+1}$  into three disjoint
 subset.
\begin{theor}
Any $(n+1)-$dimensional complex filiform Leibniz algebra admits a
basis $\{e_0,e_1,...,e_n\}$ called adapted, such that the table of
multiplication of the algebra has one of the following forms,
where non defined products are zero:
\end{theor}
$\bigskip FLeib_{n+1}=\left\{
\begin{array}{ll}
\lbrack e_{0},e_{0}]=e_{2}, &  \\[1mm]
\lbrack e_{i},e_{0}]=e_{i+1}, & \ \ \ \ \ \ 1\leq i\leq {n-1}, \\[1mm]
\lbrack e_{0},e_{1}]=\alpha _{3}e_{3}+\alpha _{4}e_{4}+...+\alpha
_{n-1}e_{n-1}+\theta e_{n}, &  \\[1mm]
\lbrack e_{j},e_{1}]=\alpha _{3}e_{j+2}+\alpha
_{4}e_{j+3}+...+\alpha
_{n+1-j}e_{n}, & \ \ \ \ \ \ 1\leq j\leq {n-2}, \\
\qquad \alpha _{3},\alpha _{4},...,\alpha _{n},\theta \in
\mathbb{C}. &
\end{array}%
\right. $ \\[1mm]

$ SLeib_{n+1}=\left\{
\begin{array}{ll}
\lbrack e_{0},e_{0}]=e_{2}, &  \\[1mm]
\lbrack e_{i},e_{0}]=e_{i+1}, & \  \qquad \ \ \ 2\leq i\leq {n-1}, \\[1mm]
\lbrack e_{0},e_{1}]=\beta _{3}e_{3}+\beta _{4}e_{4}+...+\beta
_{n}e_{n}, &
\\[1mm]
\lbrack e_{1},e_{1}]=\gamma e_{n}, &  \\[1mm]
\lbrack e_{j},e_{1}]=\beta _{3}e_{j+2}+\beta _{4}e_{j+3}+...+\beta
_{n+1-j}e_{n}, & \  \qquad \quad 2\leq j\leq {n-2}, \\
\qquad \beta _{3},\beta _{4},...,\beta _{n},\gamma \in \mathbb{C}.
&
\end{array}%
\right. $

\bigskip

$ TLeib_{n+1}=\left\{
\begin{array}{lll}
\lbrack e_{i},e_{0}]=e_{i+1}, \qquad  \qquad  \qquad \qquad  \qquad  \qquad \qquad  \qquad  \quad 1\leq i\leq {n-1},  \\[1mm]
\lbrack e_{0},e_{i}]=-e_{i+1}, \qquad \qquad  \qquad  \qquad \qquad  \qquad  \qquad \qquad  \   2\leq i\leq {n-1}, \\[1mm]
\lbrack e_{0},e_{0}]=b_{0,0}e_{n},  \\[1mm]
\lbrack e_{0},e_{1}]=-e_{2}+b_{0,1}e_{n}, \\[1mm]
\lbrack e_{1},e_{1}]=b_{1,1}e_{n},  \\[1mm]
\lbrack e_{i},e_{j}]=a_{i,j}^{1}e_{i+j+1}+\dots
+a_{i,j}^{n-(i+j+1)}e_{n-1}+b_{i,j}e_{n}, \quad   1\leq i<j\leq {n-1},   \\[1mm]
\lbrack e_{i},e_{j}]=-[e_{j},e_{i}], \qquad  \qquad  \qquad \qquad  \qquad  \qquad \qquad  \quad \ \  1\leq i<j\leq n-1,  \\[1mm]
\lbrack e_{i},e_{n-i}]=-[e_{n-i},e_{i}]=(-1)^{i}b_{i,n-i} e_{n}, \\[1mm]
\emph{\rm {where} $a_{i,j}^k, b_{i,j} \in \mathbb{C}$ and
$b_{i,n-i}=b$ whenever $1\leq i\leq {n-1},$ and $b=0$ for even
$n.$}
\end{array}%
\right. $

According to this theorem, the isomorphism problem inside of each
class can be studied separately. The classes $FLeib_{n},SLeib_{n}$
in low dimensional cases have been considered in \cite{RS1},
\cite{RS2}. The general methods of classification for $Leib_{n}$
has been given in \cite{BR1}, \cite{BR2} and \cite{OR}
(\footnote{Combined version of \cite{BR1}, \cite{BR2} and
\cite{OR} will appear in \emph{Communications in Algebra.}}). This
paper deals with the classification problem of low-dimensional
cases of $TLeib_{n}.$

Observe that the class of $n-$dimensional filiform Lie algebras is
in $TLeib_n.$

\begin{de}
Let $\{e_0,e_1,...,e_n\}$ be an adapted basis of $L\in
TLeib_{n+1}.$ Then a nonsingular linear transformation
$f:L\rightarrow L$ is said to be adapted, if the basis
 $ \{f(e_0),f(e_1),...,f(e_n)\} $ is adapted.
\end{de}

The set of all adapted elements of $GL_{n+1}$ forms a subgroup and
it is denoted by $G_{ad}.$ The following proposition specifies
elements of $G_{ad}.$
\begin{pr}
Let $f\in G_{ad}.$ Then $f$ can be represented as follows:
\begin{eqnarray*}f(e_0)=e_{0}^{\prime }&=&\sum_{i=0}^{n}A_{i}e_{i},\\[1mm]
f(e_1)=e_{1}^{\prime }&=&\sum_{i=1}^{n}B_{i}e_{i}, \\[1mm]
f(e_i)=e_{i}^{\prime }&=&[f(e_{i-1}),f(e_0)], \ \ \ \ \ \ \ \ \
2\leq i \leq n,
\end{eqnarray*} $A_0, A_i, B_j,\ (i,j=1,...,n)$ {\em are complex numbers and } $A_0\,B_1(A_0+A_1b)\neq0.$

\begin{proof}\emph{}

Since a filiform Leibniz algebra is 2-generated (see Theorem 1.1),
it is sufficient to consider the adapted action of $f$ on the
generators $e_0, e_1:$

$$\ds f(e_0)=e_{0}^{\prime }=\sum_{i=0}^{n}A_{i}e_{i},\ \
f(e_1)=e_{1}^{\prime }=\sum_{i=0}^{n}B_{i}e_{i}.$$ \\ Then $\ds
f(e_i)=[f(e_{i-1}),f(e_0)]=A^{i-2}_0(A_1B_0-A_0B_1)e_i+\sum_{j=i+1}^{n}(*)e_j,\
\ \ 2\leq i \leq n.$ \\ Note that, $A_0\neq0,$ and $A_1B_0-A_0B_1
\neq0,$ otherwise $f(e_n)=0.$ The condition
$A_0\,B_1(A_0+A_1b)\neq0$ appears naturally, since $f$ is not
singular.

 Let now consider $ \ds
[f(e_1),f(e_2)]=B_0\,(A_1\,B_0-A_0\,B_1)\,e_3+\sum_{j=4}^n
(*)e_j.$ Then for the basis $\{f(e_0),f(e_1),...,f(e_n)\}$ to be
adapted $B_0(A_1B_{0}-A_{0}B_1)$ must be zero. However, according to the
observation above, $A_1B_0-A_0B_1 \neq0.$ Therefore, $B_0=0.$
\end{proof}
\end{pr}

\section{ The description of $TLeib_{n}$,\ $n=5,6.$}

\subsection{ 5-dimensional case}

\bigskip  In this section we deal with the class $TLeib_{5}.$ By virtue of Theorem 1.1. we can represent $TLeib_{5}$
as follows:
\begin{equation*}
\ \ TLeib_{5}=\left\{
\begin{array}{llll}
\lbrack e_{i},e_{0}]=e_{i+1},  \ \ \qquad  \qquad  1\leq i\leq {3}, &  \\[1mm]
\lbrack e_{0},e_{i}]=-e_{i+1}, \qquad  \qquad 2\leq i\leq {3}, &  \\[1mm]
\lbrack e_{0},e_{0}]=b_{0,0}e_{4}, &  &  \\[1mm]
\lbrack e_{0},e_{1}]=-e_{2}+b_{0,1} e_{4}, &  &  \\[1mm]
\lbrack e_{1},e_{1}]=b_{1,1} e_{4}, &  &  \\[1mm]
\lbrack e_{1},e_{2}]=-[e_{2},e_{1}]=b_{1,2} e_{4}, && \\
b_{0,0},b_{0,1},b_{1,1},b_{1,2} \in \mathbb{C}. \end{array}%
\right.
\end{equation*}
Further, the elements of $TLeib_{5}$ will be denoted by
$L(\alpha)=L(b_{0,0}, b_{0,1},b_{1,1},b_{1,2}).$

\begin{theor}
\bigskip\ \emph{(}\textbf{Isomorphism criterion for $TLeib_{5}$}\emph{)} Two algebras
$L(\alpha )$ and $L(\alpha^\prime)$ from $TLeib_{5}$ are
isomorphic, if and only if there exist complex numbers $
A_{0},A_{1}$ and $B_{1}:$ $A_{0}\,B_{1}\neq 0$ and the following
conditions hold:
\begin{eqnarray}b_{0,0}^\prime&=&{\frac {{A^{2}_{{0}}}b_{{0,0}}+A_{{0}}A_{{1}}b_{{0,1}}+{A^{2}_{{
1}}}b_{{1,1}}}{{A^{3}_{{0}}}B_{{1}}}},
\\
b_{0,1}^\prime&=&{\frac
{A_{{0}}b_{{0,1}}+2\,A_{{1}}b_{{1,1}}}{{A^{3}_{{0}}}} }
,\\
b_{1,1}^\prime&=&\frac{B_1b_{1,1}}{A^{3}_0},\\
b_{1,2}^{\prime }&=&{\frac {B_{{1}}b_{{1,2}}}{{A^{2}_{{0}}}}}.
\end{eqnarray}

\begin{proof} Part  `` if ''. Let
$L_1$ and $L_2$ from $TLeib_5$ be isomorphic: $f:L_1 \cong L_2.$
We choose the corresponding adapted bases $\{e_0, e_1, e_2, e_3,
e_4\}$ and $\{e_0', e_1', e_2', e_3', e_4'\}$ in $L_1$ and $L_2,$
respectively. Then, in these bases the algebras will be presented
as $L(\alpha)$ and $L(\alpha^\prime),$ where $\alpha=(b_{0,0}, b_{0,1},b_{1,1},b_{1,2}),$ and $\alpha'=(b'_{0,0}, b'_{0,1},b'_{1,1},b'_{1,2}).$

According to Proposition 1.1 one has:
\begin{eqnarray}
e_0'=f(e_0)&=&A_0e_0+A_1e_1+A_2e_2+A_3e_3+A_4e_4,\\[1mm]
e_1'=f(e_1)&=&B_1e_1+B_2e_2+B_3e_3+B_4e_4.\nonumber\\[1mm]
 \mbox{Then we get}&& \nonumber \\[1mm]
e_2'=f(e_2)&=&[f(e_1),f(e_0)]=A_{{0}}B_{{1}} e_{{2}}+ A_{{0}}
B_{{2}}e_{{3}}+\left( A_{{0}}B_{{3}}+A_{{1}}B
_{{1}}b_{1,1}+(A_{{2}}B_{{1}}-A_{{1}}B_{{2}})b_{1,2} \right)
e_{{4}}, \nonumber\\[1mm]
e_3'=f(e_3)&=&[f(e_2),f(e_0)]=A^2_0B_1e_3+(A^2_0B_2-A_0A_1B_1b_{1,2})e_4,\nonumber\\[1mm]
e_4'=f(e_4)&=&[f(e_3),f(e_0)]=A^3_0B_1e_4.\nonumber
\end{eqnarray}

By using the table of multiplications one finds the relations
between the coefficients $ b_{0,0}, b_{0,1},b_{1,1},b_{1,2}$ and
$b^\prime_{0,0}, b^\prime_{0,1},b^\prime_{1,1},b^\prime_{1,2}.$
First, consider the equality \ $ \ds [f(e_0),f(e_0)]=b'_{0,0}
f(e_4),$ we get equation  $(1)$ and from the equality
 $ \ds [f(e_1),f(e_0)]+[f(e_0),f(e_1)]=b^\prime_{0,1}f(e_4) $ \
we have $(2),$ and $[f(e_1),f(e_1)]=b'_{1,1}\,f(e_4) $ gives
$(3).$ Finally, the equality $(4)$ comes out from $\ds
[f(e_1),f(e_2)]=b'_{1,2}f(e_4).$

\textquotedblleft Only if\textquotedblright \ \ part.

Let the equalities (1)--(4) hold. Then, the base change (5) above
is adapted and it transforms $L(\alpha)$ into $L(\alpha').$
Indeed,
\begin{eqnarray*}
[e'_0,e'_0]&=&\left[\sum^4_{i=0}A_{{i}}e_i,\sum^4_{i=0}A_{{i}}e_i\right]\\[1mm]
&=&{A^{2}_{{0}}}[e_0,e_0]+A_{0}A_{1}[e_0,e_1]+A_{0}A_{1}[e_1,e_0]+A^2_{1}[e_1,e_1]\\[1mm]
&=& \left( {A^{2}_{{0}}}b_{0,0}+A_{{0}}A_{{1}}b_{{0,1}}+{A^{2}_{{
1}}}b_{{1,1}} \right)
e_{{4}}=b'_{0,0}{A^{3}_{{0}}}B_{{1}}e_4=b'_{0,0}e'_4.\\[1mm]
&&\\[1mm]
[e'_0,e'_1]&=&\left[\sum^4_{i=0}A_{{i}}e_i,\sum^4_{i=1}B_{{i}}e_i\right]\\[1mm]
&=& -(A_{{0
}}B_{{1}}e_{{2}}+A_{{0}}B_{{2}}e_{{3}}+(A_{{1}}B_{{1}}b_{{1,1}}+A_{{2}}B_{{1}}b_{{1,2}}-A_{{1}}B_{{2}}b_{{1,2}
}+A_{{0}}B_{{3}} )e_4)\\ &&+B_{{1}} \left(
b_{{0,1}}A_{{0}}+2\,A_{{1}}B_{{1,1}} \right)e_{{4}}
\\[1mm]
&=& -e'_2+A^3_0 B_1b'_{0,1}e_4= -e'_2+b'_{0,1}e'_4.
\end{eqnarray*}
By the same manner one can prove that $ [e'_1,e'_1]=b'_{1,1}e'_4, $ $
[e'_1,e'_2]=b'_{1,2}\,e'_4$ and the other products are zero.

\end{proof}
\end{theor}
For the simplification purpose, we establish the following
notation: $\ds \Delta=4b_{0,0}b_{1,1}-b^{2}_{0,1}.$

Now, we list the isomorphism classes of algebras from $TLeib_{5}.$

Represent $TLeib_{5}$ as a disjoint union of the following
subsets: \textbf{$ \ds TLeib_{5}=\bigcup\limits_{i=1}^{9}\ds
U_{5}^{i},$}\ \ where

$U_{5}^{1}=\{L(\alpha)\in TLeib_{5}\ :b_{1,1}\neq 0,\ b_{1,2} \neq
0\};$

$U_{5}^{2}=\{L(\alpha)\in TLeib_{5}\ :b_{1,1}\neq 0,\ b_{1,2} =0,\
\Delta \neq 0\};$

$U_{5}^{3}=\{L(\alpha)\in TLeib_{5}\ :b_{1,1}\neq 0,\ b_{1,2}
=\Delta = 0\};$

$U_{5}^{4}=\{L(\alpha)\in TLeib_{5}\ :b_{1,1}=0,\ b_{0,1}\neq 0,\
b_{1,2} \neq 0\};$

$U_{5}^{5}=\{L(\alpha)\in TLeib_{5}\ :b_{1,1}=0,\ b_{0,1}\neq 0,\
b_{1,2}=0\};$

$U_{5}^{6}=\{L(\alpha)\in TLeib_{5}\ :b_{1,1}=b_{0,1}=0,\
b_{0,0}\neq 0,\ b_{1,2}\neq 0\};$

$U_{5}^{7}=\{L(\alpha)\in TLeib_{5}\ :b_{1,1}=b_{0,1}=0,\
b_{0,0}\neq 0,\ b_{1,2}=0\};$

$U_{5}^{8}=\{L(\alpha)\in TLeib_{5}\ :b_{1,1}=b_{0,1}=b_{0,0}= 0,\
b_{1,2}\neq 0\};$

$U_{5}^{9}=\{L(\alpha)\in TLeib_{5}\
:b_{1,1}=b_{0,1}=b_{0,0}=b_{1,2}=0\}.$\\
\bigskip

\begin{pr}\emph{}
\begin{enumerate}
\item  Two algebras $L(\alpha )$ and $L(\alpha^{\prime})\ $from
$U^1_{5}$ are isomorphic, if and only if\ \  $ \ds \left( \frac{
b^{\prime }_{1,2} }{b^{\prime }_{1,1}}\right) ^{4}\Delta ^{\prime
}=\left( \frac{b_{1,2} }{b_{1,1} }\right) ^{4}\Delta .$

\item For any $\lambda$ from $ \mathbb{C},$ there exists $L(\alpha
)\in U^1_5: \ds \left( \frac{b_{1,2} }{b_{1,1} }\right)
^{4}\Delta=\lambda$.
\end{enumerate}
\begin{proof}
\begin{enumerate}

\item\qquad `` $\Rightarrow$ ''.\ \ Let $L(\alpha )$ and
$L(\alpha^{\prime})\ $ be isomorphic. Then, due to Theorem  2.1 it
is easy to see that $ \ds \left( \frac{ b^{\prime }_{1,2}
}{b^{\prime }_{1,1}}\right) ^{4}\Delta ^{\prime }=\left(
\frac{b_{1,2} }{b_{1,1} }\right) ^{4}\Delta .$

`` $\Leftarrow $ ''.\ \ Let the
equality $ \ds \left( \frac{ b^{\prime }_{1,2} }{b^{\prime
}_{1,1}}\right) ^{4}\Delta ^{\prime }=\left( \frac{b_{1,2}
}{b_{1,1} }\right) ^{4}\Delta$ \ hold. Consider the base change
(5) above with $\ds A_0=\frac{b_{1,1}}{b_{1,2}},$ \ $\ds
A_1=-\frac{b_{0,1}}{2b_{1,2}}$ \ and
 \ $\ds B_1=\ds \frac{b^2_{1,1}}{b^3_{1,2}}.$ This changing leads
$L(\alpha)$ into $\ds L\left(\left( \frac{b_{1,2} }{b_{1,1}
}\right) ^{4}\Delta,0,1,1\right).$ An analogous base change
transforms $L(\alpha')$ into $\ds L\left(\left( \frac{b'_{1,2}
}{b'_{1,1} }\right) ^{4}\Delta',0,1,1\right).$

Since  $ \ds \left( \frac{ b^{\prime }_{1,2} }{b^{\prime
}_{1,1}}\right) ^{4}\Delta ^{\prime }=\left( \frac{b_{1,2}
}{b_{1,1} }\right) ^{4}\Delta ,$ then $L(\alpha )$ is isomorphic
to
$L(\alpha^{\prime}).$\\
 \item\qquad Obvious.
\end{enumerate}

\end{proof}
\end{pr}

\begin{pr}\emph{}
The subsets   $U_{5}^{2},\ U_{5}^{3},\ U_{5}^{4},\ U_{5}^{5},\
U_{5}^{6},\ U_{5}^{7},\ U_{5}^{8}$ and $U_{5}^{9}$ are single
orbits with representatives $L(1,0,1,0),\ L(0,0,1,0),\
L(0,1,0,1),\ L(0,1,0,0),\ L(1,0,0,1),\ L(1,0,0,0),\ L(0,0,0,1)$
and $L(0,0,0,0),$ respectively.

\end{pr}

\begin{proof}
To prove it, we give the appropriate values of $A_0,A_1$ and $B_1$
in the base change (as for other $A_i,\ B_j,\ i,j=2,3,4$ they are
any, except where specified otherwise).

For $U_{5}^{2}:$

$e'_0=A_{{0}}e_{{0}}+A_{{1}}e_{{1}}+A_{{2}}e_{{2}}+A_{{3}}e_{{3}}+A_{{4}}e_{
{4}},
 $

$e'_1=B_{{1}}e_{{1}}+B_{{2}}e_{{2}}+B_{{3}}e_{{3}}+B_{{4}}e_{{4}},$

$e'_2=A_{{0}}B_{{ 1}}e_{{2}}+A_{{0}}B_{{2}}e_{{3}}+\left(
A_{{1}}B_{{1}}b_{1,1}+A_{{0}}B_{{3}} \right) e_{{4}}, $

$e'_3={A^{2}_{{0}}}B_{{1}}e_{{3}}+{A^{2}_{{0}}}B_{{2}}e_{{4}},$

$e'_4={A_{{0}}}^{3}B_{{1}}e_{{4}},$ \\ where $\ds
A_0=\frac{\Delta^{\frac{1}{4}}}{\sqrt{2}} \ \mathrm{,} \
A_{1}=-\frac{b_{0,1}\,\Delta^{\frac{1}{4}}}{2\sqrt{2}b_{1,1}} \
\mathrm{and} \ B_1=
\frac{\Delta^{\frac{3}{4}}}{2\sqrt{2}b_{1,1}}.$

For $U_{5}^{3}:$

$e'_0=A_{{0}}e_{{0}}+A_{{1}}e_{{1}}+A_{{2}}e_{{2}}+A_{{3}}e_{{3}}+A_{{4}}e_{
{4}},
 $

$e'_1=B_{{1}}e_{{1}}+B_{{2}}e_{{2}}+B_{{3}}e_{{3}}+B_{{4}}e_{{4}},$

$e'_2=A_{{0}}B_{{ 1}}e_{{2}}+A_{{0}}B_{{2}}e_{{3}}+\left(
A_{{1}}B_{{1}}b_{1,1}+A_{{0}}B_{{3}} \right) e_{{4}}, $

$e'_3={A^{2}_{{0}}}B_{{1}}e_{{3}}+{A_{{0}}}^{2}B_{{2}}e_{{4}},$

$e'_4={A_{{0}}}^{3}B_{{1}}e_{{4}},$ \\ where $A_0\in
\mathbb{C^*},\ A_1=\ds -{\frac {A_{{0}}b_{{0,1}}}{2\,b_{{1,1}}}}$
and $B_1=\ds {\frac {A^3_0}{b_{1,1}}}.$

For $U_{5}^{4}:$

$e'_0=A_{{0}}e_{{0}}+A_{{1}}e_{{1}}+A_{{2}}e_{{2}}+A_{{3}}e_{{3}}+A_{{4}}e_{
{4}}, $

$e'_1=B_{{1}}e_{{1}}+B_{{2}}e_{{2}}+B_{{3}}e_{{3}}+B_{{4}}e_{{4}},$

$e'_2=A_{{0}}B_{{1}}e_{{2}}+A_{{0}}B_{{2}}e_{{3}}+(A_{{0}}B_{{3}}+(A_{{2}}B_{{1}}-A_{{1}}B_{{2}})b_{{1,2}})e_{{4}},$

$e'_3={A^{2}_{{0}}}B_{{1}}e_{{3}}+({A^{2}_{{0}}}B_{{2}}-A_{{1}}A_{{0}}B_{{1}}b_{{1,2}})e_{{4}},$

$e'_4={A_{{0}}}^{3}B_{{1}}e_{{4}},$ \\ where $\ds A_0^2=b_{0,1},\
A_1=-{\frac {A_0b_{0,0}}{b_{0,1}}}$ and $B_1=\ds {\frac
{{A^{2}_{{0}}}}{b_{{1,2}}}}.$

For $U_{5}^{5}:$

$e'_0=A_{{0}}e_{{0}}+A_{{1}}e_{{1}}+A_{{2}}e_{{2}}+A_{{3}}e_{{3}}+A_{{4}}e_{
{4}}, $

$e'_1=B_{{1}}e_{{1}}+B_{{2}}e_{{2}}+B_{{3}}e_{{3}}+B_{{4}}e_{{4}},$

$e'_2= A_{{0}}B_{{1}}e_{{2}}+A_{{0 }}B_{{2}}e_{{3}}+A_{{0}}B_{{3}}
e_{{4}},
 $

$e'_3={A^{2}_{{0}}}B_{{1}}e_{{3}}+{A^2_{{0}}}B_{{2}} e_{{4}},$

$e'_4={A^{3}_{{0}}}B_{{1}}e_{{4}},$ \\  where $A^2_0=\ds b_{0,1},\
A_1=\ds -{\frac {b_{{0,0}}}{\sqrt {b_{{0,1}}}}}$ and $B_1\in
\mathbb{C^*}.$

For $U_{5}^{6}:$

$e'_0=A_{{0}}e_{{0}}+A_{{1}}e_{{1}}+A_{{2}}e_{{2}}+A_{{3}}e_{{3}}+A_{{4}}e_{
{4}}, $

$e'_1=B_{{1}}e_{{1}}+B_{{2}}e_{{2}}+B_{{3}}e_{{3}}+B_{{4}}e_{{4}},$

$e'_2=A_{{0}}B_{{1}}e_{{2}}+A_{{0}}B_{{2}}e_{{3}}+(A_{{0}}B_{{3}}+(A_{{2}}B_{{1}}-A_{{1}}B_{{2}})b_{{1,2}})e_{{4}},$

$e'_3={A^{2}_{{0}}}B_{{1}}e_{{3}}+({A^{2}_{{0}}}B_{{2}}-A_{{1}}A_{{0}}B_{{1}}b_{{1,2}})e_{{4}},$

$e'_4={A_{{0}}}^{3}B_{{1}}e_{{4}},$ \\ where
 $A^3_0=\ds b_{0,0}b_{1,2},\ A_1\in \mathbb{C}$
and $B_1=\ds {\frac
{b^{\frac{2}{3}}_{0,0}}{{b^{\frac{1}{3}}_{1,2}}}}.$

 For $U_{5}^{7}:$

$e'_0=A_{{0}}e_{{0}}+A_{{1}}e_{{1}}+A_{{2}}e_{{2}}+A_{{3}}e_{{3}}+A_{{4}}e_{
{4}}, $

$e'_1=B_{{1}}e_{{1}}+B_{{2}}e_{{2}}+B_{{3}}e_{{3}}+B_{{4}}e_{{4}},$

$e'_2= A_{{0}}B_{{1}}e_{{2}}+A_{{0 }}B_{{2}}e_{{3}}+A_{{0}}B_{{3}}
e_{{4}},
 $

$e'_3={A^{2}_{{0}}}B_{{1}}e_{{3}}+{A^2_{{0}}}B_{{2}} e_{{4}},$

$e'_4={A^{3}_{{0}}}B_{{1}}e_{{4}},$ \\  where $ A_0\in
\mathbb{C^*},\ A_1\in \mathbb{C}$ and $ B_1=\ds {\frac
{b_{0,0}}{A_0}}.$

For $U_{5}^{8}:$

$e'_0=A_{{0}}e_{{0}}+A_{{1}}e_{{1}}+A_{{2}}e_{{2}}+A_{{3}}e_{{3}}+A_{{4}}e_{
{4}}, $

$e'_1=B_{{1}}e_{{1}}+B_{{2}}e_{{2}}+B_{{3}}e_{{3}}+B_{{4}}e_{{4}},$

$e'_2=A_{{0}}B_{{1}}e_{{2}}+A_{{0}}B_{{2}}e_{{3}}+(A_{{0}}B_{{3}}+(A_{{2}}B_{{1}}-A_{{1}}B_{{2}})b_{{1,2}})e_{{4}},$

$e'_3={A^{2}_{{0}}}B_{{1}}e_{{3}}+({A^{2}_{{0}}}B_{{2}}-A_{{1}}A_{{0}}B_{{1}}b_{{1,2}})e_{{4}},$

$e'_4={A_{{0}}}^{3}B_{{1}}e_{{4}},$ \\ where $ A_0\in
\mathbb{C^*},\ A_1\in \mathbb{C}$ and $B_1={A^2_{{0}}}.$

For $U_{5}^{9}:$

$e'_0=A_{{0}}e_{{0}}+A_{{1}}e_{{1}}+A_{{2}}e_{{2}}+A_{{3}}e_{{3}}+A_{{4}}e_{
{4}}, $

$e'_1=B_{{1}}e_{{1}}+B_{{2}}e_{{2}}+B_{{3}}e_{{3}}+B_{{4}}e_{{4}},$

$e'_2= A_{{0}}B_{{1}}e_{{2}}+A_{{0 }}B_{{2}}e_{{3}}+A_{{0}}B_{{3}}
e_{{4}},
 $

$e'_3={A^{2}_{{0}}}B_{{1}}e_{{3}}+{A^2_{{0}}}B_{{2}} e_{{4}},$

$e'_4={A^{3}_{{0}}}B_{{1}}e_{{4}},$ \\  where $A_0,B_1 \in
\mathbb{C^*}$ and  $A_1 \in \mathbb{C}.$
\end{proof}

\subsection{\protect\bigskip 6-dimensional case}
\qquad This section is devoted to the description of $TLeib_{6}$.
This class can be represented by the following table of
multiplication:

 $\ \ TLeib_{6}=\left\{
\begin{array}{llll}
\lbrack e_{i},e_{0}]=e_{i+1}, \ \ \qquad \qquad 1\leq i\leq {4}, &  \\[1mm]
\lbrack e_{0},e_{i}]=-e_{i+1},\qquad \qquad 2\leq i\leq {4}, &  \\[1mm]
\lbrack e_{0},e_{0}]=b_{0,0}e_{5},\ \ \left[ e_{0},e_{1}%
\right] =-e_{2}+b_{0,1}e_{5},\ \ \left[ e_{1},e_{1}\right]
=b_{1,1} e_{5}, &  &  \\[1mm]
\lbrack e_{1},e_{2}]=-[e_2,e_1]=b_{1,2}
e_{4}+b_{1,3}e_{5}, &  &  \\[1mm]
[e_1,e_3]=-[e_{3},e_{1}]=b_{1,2}
e_{5}, &  &  \\[1mm]
[e_{1},e_{4}]=-[e_{4},e_{1}]=-[e_{2},e_{3}]=[e_{3},e_{2}]=-b_{2,3}e_{5}.
\end{array}%
\right. $

The elements of $TLeib_{6}$ will be denoted by $L(\alpha ),$ where
$\alpha=(b_{0,1} ,b_{1,1} ,b_{1,2},b_{1,3},b_{2,3}).$
\begin{theor}\emph{(}\textbf{Isomorphism criterion for $TLeib_{6}$}\emph{)} Two filiform
Leibniz algebras $L(\alpha)$ and $L(\alpha')$ from $TLeib_6$ are
isomorphic, iff there exist $ A_0,A_1,B_1,B_2,B_3\in \mathbb{C}:$
such that $A_0B_1(A_0+A_1\,b_{2,3})\neq 0$ and the following
equalities hold:

\begin{eqnarray}b_{0,0}^\prime&=&\frac{A_0^2b_{0,0}+A_0A_1b_{0,1}+A_1^2b_{1,1}}{A_0^{3}B_1(A_0+A_1\,b_{2,3})},\nonumber
\\
b_{0,1}^\prime&=&\frac{A_0b_{0,1}+2A_1b_{1,1}}{A_0^{3}(A_0+A_1\,b_{2,3})}\nonumber,\\
b_{1,1}^\prime&=&\frac{B_1b_{1,1}}{A_0^{3}(A_0+A_1\,b_{2,3})},\\
b_{1,2}^{\prime }&=& \frac{B_1 \,b_{1,2}}{A_0^2},\nonumber\\
 b_{1,3}^{\prime }&=&{\frac {
2\,A_{{0}}A_{{1}}{B^{2}_{{1}}}{b^{2}_{{1,2}}}+{A^{2}_{{0}}}{B^{2}_{{1}}}
b_{{1,3}}+\left({A^{2}_{{0}}}
\left(-2\,B_{{1}}B_{{3}}+{B^{2}_{{2}}}\right)+{A^{2}_{{1}}}{B^{2}_
{{1}}}{b^{2}_{{1,2}}} \right) b_{{2,3}} }{{A^{4}_{{0}}}B_1 \left(
A_{{0}}+A_{{1}}b_{2,3} \right) }}
 ,\nonumber\\
b_{2,3}^{\prime }&=& {\frac
{B_1\,b_{2,3}}{A_0+A_1\,b_{2,3}}}\nonumber.
\end{eqnarray}
\end{theor} The proof is the similar to that of
Theorem 2.1.

Represent $TLeib_{6}\ $as a union of the following subsets: \ \


$U_{6}^{1}=\{L(\alpha)\in TLeib_{6}\ :b_{2,3}\neq 0,\ b_{1,1} \neq
0\};$

$U_{6}^{2}=\{L(\alpha)\in TLeib_{6}\ :b_{2,3}\neq 0,\ b_{1,1}=0,\
b_{0,1}\neq 0\};$

$U_{6}^{3}=\{L(\alpha)\in TLeib_{6}\ :b_{2,3}\neq 0,\ b_{1,1}=
b_{0,1}=0,\ b_{1,2}\neq 0,\ b_{0,0}\neq0\};$

$U_{6}^{4}=\{L(\alpha)\in TLeib_{6}\ :b_{2,3}\neq 0,\ b_{1,1}=
b_{0,1}=0,\ b_{1,2}\neq 0,\ b_{0,0}=0\};$

$U_{6}^{5}=\{L(\alpha)\in TLeib_{6}\ :b_{2,3}\neq 0,\ b_{1,1}=
b_{0,1}=b_{1,2}=0,\ b_{0,0}\neq0\};$

$U_{6}^{6}=\{L(\alpha)\in TLeib_{6}\ :b_{2,3}\neq 0,\ b_{1,1}=
b_{0,1}=b_{1,2}=b_{0,0}=0\};$


$U_{6}^{7}=\{L(\alpha)\in TLeib_{6}\ :b_{2,3}=0,\ b_{1,2}\neq0,\
b_{1,1}\neq 0\};$

$U_{6}^{8}=\{L(\alpha)\in TLeib_{6}\ :b_{2,3}=0,\ b_{1,2}\neq0,\
b_{1,1}= 0,\ b_{0,1}\neq 0\};$

$U_{6}^{9}=\{L(\alpha)\in TLeib_{6}\ :b_{2,3}=0,\ b_{1,2}\neq0,\
b_{1,1}=b_{0,1}=0,\ b_{0,0}\neq 0\};$

$U_{6}^{10}=\{L(\alpha)\in TLeib_{6}\ :b_{2,3}=0,\ b_{1,2}\neq0,\
b_{1,1}=b_{0,1}=b_{0,0}=0\};$

$U_{6}^{11}=\{L(\alpha)\in TLeib_{6}\ :b_{2,3}=b_{1,2}=0,\
b_{1,1}\neq0,\ b_{1,3}\neq0\};$

$U_{6}^{12}=\{L(\alpha)\in TLeib_{6}\ :b_{2,3}=b_{1,2}=0,\
b_{1,1}\neq0,\ b_{1,3}=0,\ \Delta \neq 0\};$

$U_{6}^{13}=\{L(\alpha)\in TLeib_{6}\ :b_{2,3}=b_{1,2}=0,\
b_{1,1}\neq0,\ b_{1,3}=\Delta=0\};$

$U_{6}^{14}=\{L(\alpha)\in TLeib_{6}\ :b_{2,3}=b_{1,2}=b_{1,1}=0,\
b_{0,1}\neq 0,\ b_{1,3} \neq 0\};$

$U_{6}^{15}=\{L(\alpha)\in TLeib_{6}\ :b_{2,3}=b_{1,2}=b_{1,1}=0,\
b_{0,1}\neq 0,\ b_{1,3}=0\};$

$U_{6}^{16}=\{L(\alpha)\in TLeib_{6}\
:b_{2,3}=b_{1,2}=b_{1,1}=b_{0,1}=0,\ b_{0,0}\neq0,\
b_{1,3}\neq0\};$

$U_{6}^{17}=\{L(\alpha)\in TLeib_{6}\
:b_{2,3}=b_{1,2}=b_{1,1}=b_{0,1}=0,\ b_{0,0}\neq0,\ b_{1,3}=0\};$

$U_{6}^{18}=\{L(\alpha)\in TLeib_{6}\
:b_{2,3}=b_{1,2}=b_{1,1}=b_{0,1}=b_{0,0}=0,\ b_{1,3}\neq0\};$

$U_{6}^{19}=\{L(\alpha)\in TLeib_{6}\
:b_{2,3}=b_{1,2}=b_{1,1}=b_{0,1}=b_{0,0}=b_{1,3}=0\}.$\\

\begin{pr}\emph{}
\begin{enumerate}
\item Two algebras $L(\alpha)$ and $L(\alpha')$ from $U_{6}^{1}$
are isomorphic, if and only if

$ \ds \left(\frac
{b^\prime_{2,3}}{2\,b'_{{1,1}}-b'_{{0,1}}b'_{2,3}
}\right)^2\,\Delta^\prime=\ds \left(\frac
{b_{2,3}}{2\,b_{{1,1}}-b_{{0,1}}b_{2,3} }\right)^2\,\Delta$ \ and

$\ds{\frac { \left( 2\,b'_{{1,1}}-b'_{{2,3}}b'_{{0,1}} \right)
^{3}{b'^3_{{1,2} }}}{{b'^2_{{2,3}}}{b'^4_{{1,1}}}}}={\frac {
\left( 2\,b_{{1,1}}-b_{{2,3}}b_{{0,1}} \right) ^{3}{b^3_{{1,2}
}}}{{b^2_{{2,3}}}{b^4_{{1,1}}}}}.$

\item For any $\lambda_1, \lambda_2\in\mathbb{C},$ there exists
$L(\alpha)\in U_{6}^{1}:\ \ $

$ \ds \left(\frac {b_{2,3}}{2\,b_{{1,1}}-b_{{0,1}}b_{2,3}
}\right)^2\,\Delta=\lambda_1,$\ \ $\ds {\frac { \left(
2\,b_{{1,1}}-b_{{2,3}}b_{{0,1}} \right) ^{3}{b^3_{{1,2}
}}}{{b^2_{{2,3}}}{b^4_{{1,1}}}}}=\lambda_2.$
\end{enumerate}

Then orbits from the set $U_{6}^{1}$ can be parameterized as
$L\left(\lambda_1,0,1,\lambda_2,0 ,1\right),\ \lambda_1,
\lambda_2\in\mathbb{C}.$

\end{pr}

\newpage
\begin{pr}\emph{}
\begin{enumerate}
\item Two algebras $L(\alpha)$ and $L(\alpha')$ from $U_{6}^{2}$
are isomorphic, if and only if  \\
$ \ds {\frac { \left( b'_{{0,1}}-b'_{{2,3}}b'_{{0,0}} \right)
^{4}{{b'}^ {3}_{{1,2}}}}{{{b'}^{3}_{{2,3}}}{{b'}^{5}_{{0,1}}}}}
={\frac { \left( b_{{0,1}}-b_{{2,3}}b_{{0,0}} \right) ^{4}{b^
{3}_{{1,2}}}}{{b^{3}_{{2,3}}}{b^{5}_{{0,1}}}}}.
 $
 \item
For any $\lambda\in \mathbb{C},$ there exists $L(\alpha)\in
U_{6}^{2}:$\ $\ds {\frac { \left( b_{{0,1}}-b_{{2,3}}b_{{0,0}}
\right) ^{4}{b^
{3}_{{1,2}}}}{{b^{3}_{{2,3}}}{b^{5}_{{0,1}}}}}=\lambda.$
\end{enumerate} Therefore orbits from $U_{6}^{2}$ can be
parameterized as $L\left(0, 1, 0, \lambda, 0, 1\right),\ \lambda
\in \mathbb{C}$.

\end{pr}

\begin{pr}\emph{}
\begin{enumerate}
\item Two algebras $L(\alpha)$ and $L(\alpha')$ from $U_{6}^{7}$
are isomorphic, if and only if

$\ds{\frac
{4\,b'_{{0,0}}{b'^{4}_{{1,2}}}-2\,b'_{{1,3}}b'_{{0,1}}{b'^{
2}_{{1,2}}}+{b'^{2}_{{1,3}}}b'_{{1,1}}}{b'_{{1,2}}{b'^{2}_{{1,1}}}}}
={\frac {4\,b_{{0,0}}{b^{4}_{{1,2}}}-2\,b_{{1,3}}b_{{0,1}}{b^{
2}_{{1,2}}}+{b^{2}_{{1,3}}}b_{{1,1}}}{b_{{1,2}}{b^{2}_{{1,1}}}}}
$\ and \

$\ds{\frac { \left(
b'_{{0,1}}{b'^{2}_{{1,2}}}-b'_{{1,3}}b'_{{1,1}} \right) ^{
2}}{b'_{{1,2}}{b'^{3}_{{1,1}}}}}={\frac { \left(
b_{{0,1}}{b^{2}_{{1,2}}}-b_{{1,3}}b_{{1,1}} \right) ^{
2}}{b_{{1,2}}{b^{3}_{{1,1}}}}}. $ \item For any $
\lambda_1,\lambda_2\in\mathbb{C},$ there exists $L(\alpha)\in
U_{6}^{7}: $
$${\frac {4\,b_{{0,0}}{b^{4}_{{1,2}}}-2\,b_{{1,3}}b_{{0,1}}{b^{
2}_{{1,2}}}+{b^{2}_{{1,3}}}b_{{1,1}}}{b_{{1,2}}{b^{2}_{{1,1}}}}}=\lambda_1\
,\ \ {\frac { \left( b_{{0,1}}{b^{2}_{{1,2}}}-b_{{1,3}}b_{{1,1}}
\right) ^{ 2}}{b_{{1,2}}{b^{3}_{{1,1}}}}}=\lambda_2.$$
\end{enumerate}

The orbits from $U_{6}^{7}$ can be parameterized as
$L\left(\lambda_1, \lambda_2, 1, 1, 0, 0\right),\ \lambda_1,
\lambda_2\in \mathbb{C}$.

\end{pr}

\begin{pr}\emph{}
\begin{enumerate}
\item Two algebras $L(\alpha)$ and $L(\alpha')$ from $U_{6}^{8}$
are isomorphic, if and only if \\ $ \ds {\frac { \left(
2\,b'_{{0,0}}{b'^{2}_{{1,2}}}-b'_{{1,3}}b'_{{0,1}}
 \right) ^{3}}{{b'^{3}_{{1,2}}}{b'^{4}_{{0,1}}}}}={\frac { \left( 2\,b_{{0,0}}{b^{2}_{{1,2}}}-b_{{1,3}}b_{{0,1}}
 \right) ^{3}}{{b^{3}_{{1,2}}}{b^{4}_{{0,1}}}}}.
$

 \item For any
$\lambda\in\mathbb{C},$ there exists $L(\alpha)\in U_{6}^{8}: $
$\ds {\frac { \left(
2\,b_{{0,0}}{b^{2}_{{1,2}}}-b_{{1,3}}b_{{0,1}}
 \right) ^{3}}{{b^{3}_{{1,2}}}{b^{4}_{{0,1}}}}}=\lambda.$
\end{enumerate}

The orbits from the set $U_{6}^{8}$ can be parameterized as
$L\left(\lambda, 1,0,1 ,0 ,0\right),\ \lambda\in \mathbb{C}$.

\end{pr}

\begin{pr}\emph{}
\begin{enumerate}
\item Two algebras $L(\alpha)$ and $L(\alpha')$ from $U_{6}^{11}$
are isomorphic, if and only if $\ds \left(\frac {
b'_{1,3}}{b'_{1,1}}\right)^6 \Delta'= \left(\frac {
b_{1,3}}{b_{1,1}}\right)^6 \Delta.$

 \item For any
$\lambda\in \mathbb{C},$ there exists $L(\alpha)\in U_{6}^{11}:$ $
\ds \left(\frac { b_{1,3}}{b_{1,1}}\right)^6 \Delta=\lambda.$
\end{enumerate}

The orbits from $U_{6}^{11}$ can be parameterized as
$L\left(\lambda,0 ,1 ,0 ,1 ,0\right), \lambda\in \mathbb{C}$.

\end{pr}

\begin{pr}\emph{}

The subsets  $U_{6}^{3},\ U_{6}^{4},\ U_{6}^{5},\ U_{6}^{6},\
U_{6}^{9},\ U_{6}^{10},\ U_{6}^{12},\ U_{6}^{13},\ U_{6}^{14},\
U_{6}^{15},\ U_{6}^{16},\ U_{6}^{17},\ U_{6}^{18}$ and $
U_{6}^{19}$ are single orbits with representatives
$L(1,0,0,1,0,1),\ L(0,0,0,1,0,1),\
L(1,0,0,0,0,1),\ L(0,0,0,0,0,1),\ L(1,0,0,1,0,0),\\
L(0,0,0,1,0,0),\ L(1,0,1,0,0,0),\ L(0,0,1,0,0,0),\
L(0,1,0,0,1,0),\ L(0,1,0,0,0,0),\ L(1,0,0,0,1,0),\\
L(1,0,0,0,0,0),\ L(0,0,0,0,1,0)$ and $L(0,0,0,0,0,0),$
respectively.

\end{pr}

\section{Conclusion}
\begin{enumerate}
\item In $TLeib_5,$ we distinguished nine isomorphism classes (one
parametric family and eight concrete) of three dimensional Leibniz
algebras and shown that they exhaust all possible cases.
 \item In
the case of $TLeib_6,$ there are 19 isomorphism classes (five
parametric families and 14 concrete) and they exhaust all possible
cases.
\end{enumerate}

 \textbf{Remark.}
It should be pointed out that the filiform Lie algebras case is covered by $U_{5}^{8},\ U_{5}^{9}$ and

$U_{6}^{4},\ U_{6}^{6},\
U_{6}^{10},\ U_{6}^{18},\ U_{6}^{19}$ in five and six dimensional cases, respectively. Therefore, the list of filiform

Lie algebras in the paper agrees with the list given in \cite{GJKh}.\\

 \textbf{Acknowledgments.}
The first named author is grateful to Prof. B. A. Omirov for
helpful discussions. The research was supported by the research grant
06-01-04-SF0122 MOSTI(Malaysia). The authors also are grateful to the referees for
their critical reading of the manuscript, for valuable suggestions and comments in the original
version of this paper.

\end{document}